\documentclass[11pt,a4paper]{amsart}

\usepackage{amssymb, mathrsfs, amscd}
\usepackage[all]{xy}

\usepackage{graphics}
\usepackage{graphicx}
\usepackage{enumerate}
\usepackage{color}
\usepackage{a4wide}

\newtheorem{theorem}{Theorem}[section]
\newtheorem{definition}[theorem]{Definition}
\newtheorem{lemma}[theorem]{Lemma}

\newtheorem{proposition}[theorem]{Proposition}
\newtheorem{corollary}[theorem]{Corollary}

\title[Equivariant vector bundles over classifying spaces for proper
  actions]{Equivariant vector bundles over classifying spaces for
  proper actions} 

\date{\today}

\author{Dieter Degrijse and Ian J. Leary}

\thanks{The first author was supported by the Danish National Research
  Foundation through the Centre for Symmetry and Deformation (DNRF92)}  

\newcommand{\mF}{\mathcal {F}}

\newcommand{\orb}{\mathcal{O}_{\mF}G}

\newcommand{\col}{\colon\thinspace} 

\medskip

\begin{document}

\begin{abstract} Let $G$ be an infinite discrete group and let
  $\underline{E}G$ be a classifying space for proper actions of
  $G$. Every $G$-equivariant vector bundle over $\underline{E}G$ gives
  rise to a compatible collection of representations of the finite
  subgroups of $G$. We give the first examples of groups $G$ with a
  cocompact classifying space for proper actions $\underline{E}G$
  admitting a compatible collection of representations of the finite
  subgroups of $G$ that does not come from a $G$-equivariant (virtual)
  vector bundle over $\underline{E}G$.  This implies that the
  Atiyah-Hirzeburch spectral sequence computing the $G$-equivariant
  topological $\mathrm{K}$-theory of $\underline{E}G$ has non-zero
  differentials. On the other hand, we show that for right angled
  Coxeter groups this spectral sequence always collapses at the second
  page and compute the $\mathrm{K}$-theory of the classifying space of a right
  angled Coxeter group. 
\end{abstract}

\maketitle
\section{Introduction}
Let $G$ be an infinite discrete group and $\mathcal{F}$ be the family
of finite subgroups of $G$. Recall that the orbit category $\orb$ is 
a category whose objects are the transitive $G$-sets $G/H$, for all $H
\in \mathcal{F}$, and whose morphism are all $G$-equivariant maps
between the objects. A classifying space for proper 
actions of $G$, denoted by $\underline{E}G$, is a proper
$G$-CW-complex such that the fixed point set $\underline{E}G^H$ is
contractible for every $H\in \mathcal{F}$. The space $\underline{E}G$
is said to be cocompact if the orbit space $G\setminus
\underline{E}G=\underline{B}G$ is compact. Many interesting classes of
groups $G$ have cocompact models for $\underline{E}G$, for example 
cocompact lattices in Lie groups, mapping class groups of surfaces,
$\mathrm{Out}(F_n)$, CAT(0)-groups and word-hyperbolic groups. We
refer the reader to \cite{Luck2} for more examples and details.

Now assume $G$ is an infinite discrete group admitting a cocompact
classifying space for proper actions $\underline{E}G$. If \[\xi\col E
\rightarrow \underline{E}G\] is a $G$-equivariant complex vector
bundle over $\underline{E}G$ (see Definition \ref{def: vector bundle})
and $x$ is a point of $\underline{E}G$, then the fiber $\xi^{-1}(x)$
is a complex representation of the finite isotropy group $G_x$. The
connectivity of the fixed point sets of $\underline{E}G$ ensures that
these representations are compatible (see Definition \ref{def:
compatible}) with one another as $x$ and hence $G_x$
varies. Therefore, every $G$-equivariant complex vector bundle over
$\underline{E}G$ gives rise to a compatible collection of complex
representations of the finite subgroups of $G$, and hence to an
element of \[\lim_{G/H \in \orb}R(H).\] Here, $\lim_{G/H \in
\orb}R(H)$ is the limit over the orbit category $\orb$ of the
contravariant representation ring functor
\[    R(-)\col \orb \rightarrow \mathrm{Ab}\qquad G/H \mapsto R(H).  \]
Denoting the Grothendieck group of the abelian monoid of isomorphism
classes of complex $G$-vector bundles over $\underline{E}G$ by
$\mathrm{K}^{0}_G(\underline{E}G)$, one obtains a map 
\[  \varepsilon_G\col  \mathrm{K}^{0}_G(\underline{E}G) 
\rightarrow  \lim_{G/H \in \orb}R(H)  \]
that maps a formal difference of (isomorphism classes) vector bundles
(i.e.~a virtual vector bundle) to a  formal difference of (isomorphism
classes) of compatible collections of representations of the finite
subgroups of $G$. We say a compatible collection of represesentations
of the finite subgroups of $G$ can be realized as a (virtual)
$G$-equivariant vector bundle over $\underline{E}G$ if there exists a
(virtual) $G$-equivariant vector bundle over $\underline{E}G$ that
maps to this collection under $\varepsilon_G$. One can also look at
the corresponding situation for real (orthogonal) vector bundles and
real (orthogonal) representations and obtain the map 

\[  \varepsilon_G\col \mathrm{KO}^{0}_G(\underline{E}G) \rightarrow  \lim_{G/H
  \in \orb}RO(H).  \] The maps $\varepsilon_G$ are equal to the edge
homomorphisms of certain Atiyah-Hirzebruch spectral sequences
converging to $\mathrm{K}^{\ast}_G(\underline{E}G)$ and
$\mathrm{KO}^{\ast}_G(\underline{E}G)$ (see (\ref{eq: ss complex}) and
(\ref{eq: ss real})). L\"{u}ck and Oliver proved that (see Proposition
\ref{prop: rational collapse}) the map $\varepsilon_G$ (real or
complex) is rationally surjective, meaning that a high enough multiple
of every element in the target of $\varepsilon_G$ is contained in the
image of $\varepsilon_G$. In the last paragraph of
\cite[p. 596]{LuckOliver} L\"{u}ck and Oliver ask for an example of a
group $G$ admitting a cocompact classifying space for proper actions
$\underline{E}G$ such that $\varepsilon_G$ is not surjective. In
Section \ref{sec: complex} of this paper we give the first example of
such a group in the complex case.  In Section \ref{sec: real} we give
the first example of such a group in the real case. We also construct
examples of groups $G$ admitting a cocompact $\underline{E}G$ with the
following weaker property: $G$ admits a compatible collection of
representations for its finite subgroups that cannot be realized as a
$G$-vector bundle over $\underline{E}G$.  However, for these examples
we cannot exclude the possibility that there exists a virtual vector
bundle that maps to this collection of representations under
$\varepsilon_G$.  On the other hand, these examples are more explicit
and lower dimensional.

In the final section we show that for a right angled Coxeter group $W$,
every compatible collection of representations of the finite subgroups
of $W$ can be realized as a $W$-equivariant vector bundle over
$\underline{E}W$, so that the map
\[     \varepsilon_W\col \mathrm{K}^{0}_W(\underline{E}W) \rightarrow  \lim_{W/H
  \in \mathcal{O}_{\mathcal{F}}W}R(H).  \] is always
surjective. Morever, we show that this map is actually an isomorphism
and that  (see Theorem \ref{eq:
cohomology theory}) \[\mathrm{K}^{1}_W(\underline{E}W)=0.\]
Using a version of the Atiyah-Segal completion theorem for infinite
discrete groups proven by L\"{u}ck and Oliver, we use these results to
compute the complex $\mathrm{K}$-theory of $BW$, the classifying space of $W$
(see Corollary \ref{cor: ktheory}).

\section*{Acknowledgement} 
We thank the referee, whose comments on an earlier version of this
article were extremely helpful.  

\section{$G$-vector bundles and isotropy representations }
Let $G$ be a discrete group and let $\Gamma$ be a Lie group.  Let
$\mathcal{S}$ be a family of finite subgroups of $G$, i.e. any
collection of finite subgroups of $G$ that is closed under conjugation
and passing to subgroups. The orbit category
$\mathcal{O}_{\mathcal{S}}G$ is a category whose objects are the
transitive $G$-sets $G/H$, for all $H \in \mathcal{S}$, and whose
morphism are all $G$-equivariant maps between the objects.

\begin{definition}\cite[p.~590]{LuckOliver} \rm \label{def: compatible} 
Let $X$ be a $G$-CW-complex. A $(G,\Gamma)$-bundle over $X$ is a
$\Gamma$-principal bundle $p\col E\rightarrow X$, where $E$ is a left
$G$-space such that $p$ is $G$-equivariant and such that the left
$G$-action and the right $\Gamma$-action on $E$ commute. We denote the
set of isomorphism classes of $(G,\Gamma)$-bundles over $X$ by
$\mathrm{Bdl}_{(G,\Gamma)}(X)$.  For $H \in \mathcal{F}$, let
	\[   \mathrm{Rep}_{\Gamma}(H)= 
\mathrm{Hom}(H,\Gamma)/\mathrm{Inn}(\Gamma). \]
 One can consider $\mathrm{Rep}_{\Gamma}(-)$ as a contravariant functor from
$\mathcal{O}_{\mathcal{S}}G$ to $\mathrm{Sets}$. An element of the
limit
\[A=([\alpha_H])_{H \in\mathcal{S}}\in \lim_{G/H \in
	\mathcal{O}_{\mathcal{S}}G}\mathrm{Rep}_{\Gamma}(H)\] is a
called an \emph{$\mathcal{S}$-compatible collection of
  $\Gamma$-representations}. Given such an element $A$, let
$\mathcal{S}_{A}$ be the family of subgroups of $G \times \Gamma$
consisting of conjugates of the subgroups of the form
\[ \{(h,\alpha_H(h)) \ |  h \in H \}    \] 
for all $H \in \mathcal{S}$ and let $E_{\mathcal{S}}(G,A)$ be the 
universal $G\times \Gamma$-CW-complex for the family $\mathcal{S}_{A}$.
\end{definition} 

\begin{lemma}\label{lemma: uni bun} \cite[Lemma 2.4]{LuckOliver} 
For every $\mathcal{S}$-compatible collection of
$\Gamma$-representations $A=([\alpha_H])_{H \in\mathcal{S}}$ there
exists a $G$-CW-complex $B_{\mathcal{S}}(G,A)$ with isotropy in
$\mathcal{S}$ satisfying the following properties.
\begin{itemize}
\item [-] The quotient map $$\pi\col E_{\mathcal{S}}(G,A)\rightarrow 
\Gamma \setminus E_{\mathcal{S}}(G,A) =B_{\mathcal{S}}(G,A)  $$ 
is a $(G,\Gamma)$-bundle over the $G$-CW-complex $B_{\mathcal{S}}(G,A)$.	 
\item[-] The $(G,\Gamma)$-bundle $\pi\col E_{\mathcal{S}}(G,A)\rightarrow 
B_{\mathcal{S}}(G,A)$ is universal in the sense that for every $G$-CW-complex 
$X$ with isotropy in $\mathcal{S}$ there is an  isomorphism
\[  [ X , \mathcal{B}_{\mathcal{S}}(G,A)]_G \xrightarrow{\cong} 
\mathrm{Bld}_{(G,\Gamma)}(X)\]
given by pulling back the univeral bundle $\pi$ along a 
$G$-map $X \rightarrow \mathcal{B}_{\mathcal{S}}(G,A)$.
\item[-] For every $S\in \mathcal{S}$, the fixed point set 
$B_{\mathcal{F}}(G,A)^H$ is homotopy equivalent to $BC_{\Gamma}(\alpha_H)$, 
the classifying space of the centralizer of the image of $\alpha_H$ 
in $\Gamma$. 	
		
\end{itemize}

\end{lemma}

If $\Gamma=\mathrm{U}(n)$ ($\Gamma=\mathrm{O}(n)$) and
$\mathcal{S}=\mathcal{F}$, the family of all finite subgroups of $G$,
then $\mathrm{Rep}_{\Gamma}(H)$ is the set of isomorphism classes of
$n$-dimensional complex (real) representations of $H$. In this case,
an element of the limit
\[A=([\alpha_H])_{H \in\mathcal{F}}\in \lim_{G/H \in
	\mathcal{O}_{\mathcal{F}}G}\mathrm{Rep}_{\Gamma}(H)\] is a called
is called a \emph{compatible collection of complex (real) $n$-dimensional
	representations} of the finite subgroups of $G$. 
For $H \in
\mathcal{F}$, let $R(H)$ ($RO(H)$) be the complex (real)
representation ring of $H$, i.e.~the Grothendieck group of the
abelian cancellative monoid of isomorphism classes of finite
dimensional complex (real) representations of $H$. Note that
$\mathrm{Rep}_{\mathrm{U(n)}}(H)$ is naturally a subset of $R(H)$ and
$\mathrm{Rep}_{\mathrm{O(n)}}(H)$ is naturally a subset of
$RO(H)$. One can consider $R(-)$ as a functor from
$\mathcal{O}_{\mathcal{F}}G$ to $\mathrm{Ab}$. An element of the
inverse limit
\[\alpha=([\alpha_H])_{H \in\mathcal{F}}\in \lim_{G/H \in
  \mathcal{O}_{\mathcal{F}}G}R(H)\] 
is a called a \emph{compatible collection of complex virtual
representations} of the finite subgroups of $G$.  One has a 
natural embedding 
\[ \lim_{G/H \in \mathcal{O}_{\mathcal{F}}G}
\mathrm{Rep}_{\mathrm{U}(n)}(H)
\subset  \lim_{G/H \in \mathcal{O}_{\mathcal{F}}G}R(H).\] 
The analogous statements for
$\mathrm{O}(n,\mathbb{R})$ and $RO$ also hold.

Now let $X$ be a proper cocompact $G$-CW-complex, i.e.~$X$ has finite
isotropy and the orbit 
space $G\setminus X$ has a finite number of cells, such that for every
$H\in \mathcal{F}$, the fixed point set 
$X^H$ is non-empty and connected.

\begin{definition}[\cite{Segal}]\label{def: vector bundle} \rm  A complex (real)
  $G$-vector bundle over $X$ is a complex (real) vector bundle $\pi\col E
  \rightarrow X$ such that $\pi$ is $G$-equivariant and each $g\in G$
  acts on $E$ and $X$ via a bundle isomorphism. An isomorphism of
  $G$-vector bundles over $X$ is just an isomorphism of vector bundles
  that is $G$-equivariant. The set of isomorphims classes of complex
  (real) $G$-vector bundles over $X$ will be denoted by
  $\mathrm{Bdl}_{G}(X)$ ($\mathrm{OBdl}_{G}(X)$). For every $x\in X$,
  the fiber $\pi^{-1}(x)$ is denoted by $E_x$. We refer the reader to
  \cite[Section 1]{LuckOliver} and \cite[Section I.9]{Dieck} for
  elementary properties of $G$-vector bundles over proper (cocompact)
  $G$-CW complexes. 

\end{definition}

\begin{theorem}{\rm\cite[Th. 3.2 and 3.15]{LuckOliver} } \label{eq:
  cohomology theory}There exists a $2$-periodic ($8$-periodic)
  equivariant cohomology theory $\mathrm{K}^{\ast}_G(X,A)$ 
($\mathrm{KO}^{\ast}_G(X,A)$)
  on the category of proper $G$-CW-pairs such that when 
  $X$ is cocompact, $\mathrm{K}^{0}_G(X)$
  ($\mathrm{KO}^{0}_G(X)$) is the Grothendieck group of the abelian monoid of
  isomorphism classes of complex (real) $G$-vector bundles over
  $X$. In particular, for every $H \in \mathcal{F}$, $\mathrm{K}^{0}_G(G/H)$
  ($\mathrm{KO}^{0}_G(G/H)$ ) is canonically isomorphic to $R(H)$ ($RO(H)$). 
\end{theorem}
As usual (see \cite[Section 6]{LuckOliver2} and \cite[Th. 4.7]{DL}),
the skeletal filtration of $X$ induces Atiyah-Hirzebruch spectral
sequences
\begin{equation}  \label{eq: ss complex} 
E_2^{p,q}=\mathrm{H}^{p}_G(X,\mathrm{K}^{q}_G(G/-))  
\Longrightarrow \mathrm{K}_G^{p+q}(X). \end{equation}
and 
\begin{equation}\label{eq: ss real}   
E_2^{p,q}=\mathrm{H}^{p}_G(X,\mathrm{KO}^{q}_G(G/-))  
\Longrightarrow \mathrm{KO}_G^{p+q}(X) \end{equation}
where $\mathrm{H}^{p}_G(X,-)$ denotes Bredon cohomology of 
$X$ (see \cite{Bredon}).
\begin{proposition}\label{prop: rational collapse} {\rm \cite[Prop
      5.8]{LuckOliver2}} If $X$ is a cocompact $G$-CW complex then the
  spectral sequences (\ref{eq: ss complex}) and (\ref{eq: ss real})
  above rationally collapse, meaning that the images of all
  differentials in these spectral sequences consist of torsion
  elements.  
\end{proposition}
By our assumptions on $X$, the zeroth Bredon cohomology group
$\mathrm{H}^{0}_G(X,R(-))$ (resp. $\mathrm{H}^{0}_G(X,RO(H))$), equals
the limit of the functor $R(-)$ (resp. $RO(-)$), over the orbit
category $\mathcal{O}_{\mathcal{F}}G$. Consider the edge homomorphisms
\[  \varepsilon_G\col \mathrm{K}_G^{0}(X) \rightarrow 
\mathrm{H}^{0}_{G}(X,R(-))    \]
and
\[  \varepsilon_G\col \mathrm{KO}_G^{0}(X) \rightarrow 
\mathrm{H}^{0}_{G}(X,RO(-))    \]
of the spectral sequences (\ref{eq: ss complex}) and  (\ref{eq: ss
  real}). If $[\pi]$ is the isomorphism class of an $n$-dimensional
complex $G$-vector bundle $\pi\col E \rightarrow X$, then
$\varepsilon_G([\pi])$ equals  
\[ ([E_{e_H}])_{H \in \mathcal{F}} \in \lim_{G/H \in \mathcal{O}_{\mathcal{F}}G} \mathrm{Rep}_{U(n)}(H)\subset \mathrm{H}^{0}_G(X,R(-))\]
where $[E_{e_H}]$ denotes the isomorphism class in $R(H)$ of the $H$-representation $E_{e_H}$. The corresponding statement for real $G$-vector bundles also holds. Note that it follows from Proposition \ref{prop: rational collapse} that a suitable multiple of every compatible collection of (virtual) real or complex representations of the finite subgroups of $G$ is contained in the image of the edge homomorphism $\varepsilon_G$. \\

Recall that the classifying space for proper actions $\underline{E}G$ is a terminal object in the homotopy category of proper $G$-CW complexes (e.g. \cite[Th. 1.9]{Luck2}). Hence, if $X$ is any proper cocompact $G$-CW complex such that $X^H$ is non-empty and connected for each $H \in \mathcal{F}$, then there exists a $G$-map $X \rightarrow \underline{E}G$ that is unique up to $G$-homotopy and induces commutative diagrams
\[  \xymatrix{ \mathrm{K}_G^{0}(X) \ar[r] &  
\lim_{G/H \in \mathcal{O}_{\mathcal{F}}G}R(H) \\
\mathrm{K}_G^{0}(\underline{E}G) \ar[u] \ar[ur] & } 
\ \ \ \mbox{and} \ \  \xymatrix{ \mathrm{KO}_G^{0}(X) 
\ar[r] &  \lim_{G/H \in \mathcal{O}_{\mathcal{F}}G}RO(H) \\
\mathrm{KO}_G^{0}(\underline{E}G). \ar[u] \ar[ur] & } \] 
Hence, if a compatible collection $\alpha$ of virtual representations
can be realized as a virtual $G$-vector bundle over $\underline{E}G$,
it can also be realized as a virtual $G$-vector bundle over $X$.

\section{Complex vector bundles} \label{sec: complex}
The purpose of this section is to construct a group $G$ with a
cocompact classifying space for proper actions $\underline{E}G$
admitting a compatible collection of complex representations of the
finite subgroups of $G$ that cannot be realized as $G$-equivariant
virtual complex vector bundle over $\underline{E}G$, i.e.~so that the
edge homomorphism 

\[  \varepsilon_{G}\col \mathrm{K}^{0}_{G}(\underline{E}G) 
\rightarrow \lim_{G/H \in \mathcal{O}_{\mathcal{F}}G} R(H). \]
is not surjective.\\

Let $F=C_4 \rtimes C_2$ be the dihedral group of order $8$ where
$\sigma$ is generator for $C_4$ and $\varepsilon$ is a generator of
$C_2$. Let $H=\langle \sigma^2 \rangle$ be the center of $F$, which
has order two and denote the $n$-skeleton of the universal $F/H$-space
$X=E(F/H)$ by $X^n$. We let $F$ act on $X$ and $X^n$ via the
projection onto $F/H$. Consider the complex $1$-dimensional
representation
\[\lambda\col H=\langle \sigma^2 \rangle \rightarrow 
\mathrm{U}(1)=S^{1}: \sigma^2 \mapsto -1. \]
\begin{lemma} \label{lemma: image res} The isomorphism class $[\lambda]$ is contained in $R(H)^{F/H}$. For $k \in \mathbb{Z}$, the multiple $k[\lambda]$ is contained in the image of the restriction map $\mathrm{res}\col R(F)\rightarrow R(H)$ if and only if $k$ is even.

\end{lemma}

\begin{proof} Since $H$ is the center of $F$ it follows that the conjugation action of $F/H$ on $R(H)$ is trivial, hence $[\lambda] \in R(H)^{F/H}=R(H)$. One easily verifies that the representation
\[   \tau\col F\rightarrow U(2) \]
defined by 
\[\tau(\sigma)=\left(\begin{array}{cc}
 0 & i \\
i & 0
\end{array}\right) \ \ \mbox{and} \ \ \tau(\varepsilon)=\left(\begin{array}{cc}
 -1 & 0 \\
0 & 1
\end{array}\right) \]
satisifies $\mathrm{res}([\tau])=2[\lambda]$. Hence, $k[\lambda]$ is contained in the image of $\mathrm{res}$ for every even $k \in \mathbb{Z}$. Note that, as a free abelian group, $R(H)$ is generated by $[\lambda]$ and the isomorphism class of the $1$-dimensional complex trivial representation $[\mathrm{tr}]$ (e.g.~see \cite{serre}). Now suppose $k$ is odd and there exists an element $[\mu]-[\rho] \in R(F)$ such that $\mathrm{res}([\mu]-[\rho])=k[\lambda]$. There are integers $l,m,n$ and such that $\mathrm{res}([\mu])=l[\mathrm{tr}]+m[\lambda]$, $\mathrm{res}([\rho])=l[\mathrm{tr}]+n[\lambda]$ and $m-n=k$. By changing the representative of $[\mu]$, we may also assume that 

\[   \mu\col F\rightarrow  \mathrm{U}(l+m) \]
where $\mu(\sigma)$ is a diagonal matrix. Since $\mu(\sigma^2)$ has an $m$-dimensional eigenspace with eigenvalues $-1$ and an $l$-dimensional eigenspace with eigenvalue $1$, it follows that $\mu(\sigma)$ has an $s$-dimensional eigenspace with eigenvalue $i$ and a $t$-dimensional eigenspace with eigenvalue $-i$ such that $s+t=m$. Morever, $\mu(\sigma^3)$ has an $s$-dimensional eigenspace with eigenvalue $-i$ and a $t$-dimensional eigenspace with eigenvalue $i$. Since $\sigma$ and $\sigma^3$ are conjugate in $F$, it follows that $s=t$ proving that $m$ is even. A similiar argument shows that $n$ is also even. But this contradicts the fact that  $k=m-n$ is odd. Hence, there does not exist an element $[\mu]-[\rho] \in R(F)$ such that $\mathrm{res}([\mu]-[\rho])=k[\lambda]$, if $k$ is odd.

\end{proof}
The following lemma uses the notation introduced above and will be cited in the next section.
\begin{lemma} \label{lemma: line bundles}  

Every $F$-equivariant complex line bundle over $X^3$ is isomorphic to the pullback of an $F$-equivariant complex line bundle over $E(F/H)$ along the inclusion $i\col X^3 \rightarrow E(F/H)$.

\end{lemma}

\begin{proof} 
Let $\mathcal{S}$ be the family of subgroups of $F$ containing only
$H$ and the trivial subgroup. Note that isomorphism classes of
$F$-equivariant complex line bundles correspond to isomorphism classes
of $(F,S^1=\mathrm{U}(1))$-bundles. Let $\pi\col E \rightarrow X^3$ be an
$F$-equivariant complex line bundle over and let $[\alpha_H\col H
  \rightarrow \mathrm{U}(1)=S^{1}]$ be the isomorphism class in
$\mathrm{Rep}_{S^{1}}(H)$ of the $H$-representation induced on the
fibers of $\pi$. If we set $\alpha_{\{e\}} \col \{e\} \rightarrow S^{1}$,
then $A=([\alpha_K])_{K \in \mathcal{S}} \in \lim_{K \in
  \mathcal{S}}\mathrm{Rep}_{S^{1}}(K)$. It follows from Lemma
\ref{lemma: uni bun} for $\Gamma=S^{1}$, that in order to show that
$\pi$ is the pullback of an $F$-equivariant complex line bundle over
$E(F/H)$ along the inclusion $i\col X^3 \rightarrow E(F/H)$, it suffices
to show that every $F$-map from $X^3$ to $B_{\mathcal{S}}(F,A)$ can be
extended to an $F$-map from $E(F/H)$ to $B_{\mathcal{S}}(F,A)$. Here
$B_{\mathcal{S}}(F,A)$ is homotopy equivalent to
$BS^{1}=\mathbb{C}P^{\infty}$ for all $K \in \mathcal{S}$, again by
Lemma \ref{lemma: uni bun}. It follows from Bredon's equivariant
obstruction theory (see \cite[Section
  II.1]{Bredon},\cite[Th. I.5.1]{May}) that the potential obstructions
for extending such a map lie in the relative Bredon cohomology groups
$\mathrm{H}^{n+1}_{F}(E(F/H),X^3;\pi_{n}(\mathrm{B}_{\mathcal{S}}(F,A)^{-}))$
for $n\geq 3$. Since $\pi_n(\mathbb{C}P^{\infty})$ is zero unless
$n=2$, the lemma is proven.
\end{proof}

The idea for the following lemma is contained in \cite[p 596]{LuckOliver}.
\begin{lemma} \label{lemma: no virtual bundle} There exists an $n\geq 1$ such that $[\lambda]$ is not contained in the image of the edge homomorphism
\[    \mathrm{K}^{0}_F(X^n) \rightarrow R(H)^{F/H}. \]
\end{lemma}
\begin{proof}
By \cite[Theorem 5.1]{Jackowski} for $X=\{\ast\}$, $\mathcal{F}=\{e,H\}$ and $E\mathcal{F}=E(F/H)$, there are maps
\[   \alpha_{n}\col   R(F)/I^n  \rightarrow  \mathrm{K}^0_F(X^n)    \] 
that induce a map of inverse systems from $\{ R(F)/I^n \}_{n\geq 0} $
to $ \{ \mathrm{K}^0_F(X^n)\}_{n\geq 0}$ that induces an isomorphism
of pro-rings. Here $I$ is the kernel of the restriction map $R(F)
\rightarrow R(H)$. This implies that for sufficiently large $n\geq 1$
there exists a map $\beta_{1}\col\mathrm{K}^0_F(X^n) \rightarrow R(F)/I$
making the following diagram commute
\[\xymatrix{R(F)/I^n \ar[r]^{\alpha_n} \ar[dd]  & 
\mathrm{K}^0_F(X^n)  \ar[ddl]^{\beta_{1}} \ar[dd] \ar[dr]^{\varepsilon_F} &\\
& &  R(H)^{F/H} \\
R(F)/I \ar[r]^{\alpha_1} &  \mathrm{K}^0_F(X^1). \ar[ur]^{\varepsilon_F} }\]
This shows that the image of the restriction map 
\[      R(F) \rightarrow R(H)^{F/H}   \]
coincides with the image of the edge homomorphism 
\[    \mathrm{K}^{0}_F(X^n) \rightarrow R(H)^{F/H}. \]
Since $[\lambda]$ does not lie in the image of 
$\mathrm{R}(F) \rightarrow \mathrm{R}(H)^{F/H}$ by 
Lemma \ref{lemma: image res}, the lemma follows.
\end{proof}
Let $n\geq 3$. By \cite[Th.~A \& Th.~8.3]{Leary} there exists a
compact $n$-dimensional locally CAT(0)-cubical complex $T_{X^n}$
equipped with a free cellular $F/H$-action and an $F/H$-equivariant
map $t_{X^n}\col T_{X^n} \rightarrow X^{n}$ that induces an isomorphism
\begin{equation}\label{eq: equivcohom1}  \mathcal{H}^{\ast}_F(X^{n})
  \xrightarrow{\cong} \mathcal{H}^{\ast}_{F}(T_{X^n}) \end{equation}
for any equivariant cohomology theory $\mathcal{H}^{\ast}_{?}(\cdot) $
(e.g. see \cite[section 1]{Luck3}).  (We remark that
\cite[Th.~8.3]{Leary} is stated for equivariant \emph{homology}
theories, but the analogous statement holds for equivariant 
\emph{cohomology} theories by essentially the same proof.)  
The action of $F$ on $T_{X^n}$ in the above is via the
projection $F \rightarrow F/H$.  Now let $Y^n$ be the universal cover
of $T_{X^n}$ and let $\Gamma_n$ be the group of self-homeomorphisms of 
$Y^n$ that lift the action of $F/H$ on $T_{X^n}$. Since $F/H$ acts
freely on $T_{X^n}$, $\Gamma_n$ acts freely on $Y^n$. We conclude that $Y^n$ is an $n$-dimensional CAT(0)-cubical complex on which $\Gamma_n$ acts freely, cocompactly and
cellularly. Since $Y_n$ is contractible, this implies that $\Gamma_n$ is torsion-free. By construction there is a surjection $\Gamma_n\rightarrow
F/H $ whose kernel $N_n$ is the torsion-free group of deck
transformation of the covering $Y^n \rightarrow T_{X^n}$. Now define
the group $G_n$ to be the pullback of $\pi_n\col \Gamma_n \rightarrow
F/H$ along $F \rightarrow F/H$. Then $G_n$ acts on $Y^n$ via the
quotient map $G_n \rightarrow G_n/H=\Gamma_n$ and fits into the short
exact sequence
\[  1 \rightarrow N_n \rightarrow G_n \xrightarrow{p_n} F \rightarrow 1.  \]
Note that the only non-trivial finite subgroup of $G_n$ is $H\cong
C_2$ and that since $N_n$ acts freely on $Y^n$, the $G_n$-equivariant
quotient map $Y^n \rightarrow N_n \setminus Y^n= T_{X^n}$ induces an
isomorphism (\cite[Lemma 3.5]{LuckOliver})
\begin{equation}  \label{eq: equivcohom2}\mathrm{K}^{\ast}_F(T_{X^n})
  \xrightarrow{\cong} \mathrm{K}^{\ast}_{G_n}(Y^n).  \end{equation}
Applying (\ref{eq: equivcohom1}) and (\ref{eq: equivcohom2}) to the
compostion $Y^n \rightarrow T_{X^n} \rightarrow X^n$ and the
equivariant cohomology theories $\mathrm{K}^{\ast}_{?}(\cdot)$ and
$\mathrm{H}^{\ast}_{?}(\cdot,R(-))$ with $\ast=0$, we obtain a
commutative diagram
\[   \xymatrix{\mathrm{K}^0_F(X^n) \ar[r]^{\cong}
  \ar[d]^{\varepsilon_F} & \mathrm{K}^{0}_{G_n}(Y^n)
  \ar[d]^{\varepsilon_{G_n}} \\ 
R(H)^{F/H}\ar[r]^{\cong \ \ \ \ \ \ \ \ } & \lim_{G_n/S \in
  \mathcal{O}_{\mathcal{F}}G_n} R(S).}\] 
The fact that this diagram commutes can be seen as follows. Using equivariant cellular approximation, we may assume that the map $X^n \rightarrow  Y^n$ is cellular. By considering the inclusion of zero-skeleta in $n$-skeleta, naturality yields a commutative diagram 
\[   \xymatrix{\mathrm{K}^0_F(X^n) \ar[r]^{\cong}
	\ar[d] & \mathrm{K}^{0}_{G_n}(Y^n)
	\ar[d] \\ 
	\mathrm{K}^0_F(X^0)\ar[r] & \mathrm{K}^0_{G_n}(Y^0).}\] 
The edge homomorphism $\varepsilon_F\col \mathrm{K}^0_F(X^n) \rightarrow R(H)^{F/H} \subseteq \mathrm{K}^0_F(X^0) $ coincides by construction with $\mathrm{K}^0_F(X^n) \rightarrow  \mathrm{K}^0_F(X^0)$ once we restrict the codomain, and similarly for $\varepsilon_{G_n}$. Therefore, commutativity follows.\\

Since we proved in Lemma \ref{lemma: no virtual bundle} that, for $n$
large enough, the isomorphism class of $\lambda$ does not lie in the
image of the edge homomorphism
\[    \mathrm{K}^{0}_F(X^n) \rightarrow R(H)^{F/H} \]
it follows from the commutative diagram above that 
the compatible system of representations \[ (\lambda \circ
p_{n|S})_{S \in \mathcal{F}}\in \lim_{G_n/S \in
	\mathcal{O}_{\mathcal{F}}G_n}R(S)=\mathrm{H}^{0}_{\mathcal{F}}(G_n,R(-)).\]  
does not lie in the image of the edge homomorphism
\[  \varepsilon_{G_n}\col \mathrm{K}^{0}_{G_n}(Y^n) \rightarrow \lim_{G_n/S \in
  \mathcal{O}_{\mathcal{F}}G_n} R(S). \] 

Recall from \cite{BridHaef} that non-empty CAT(0)-cube complexes are
contractible and that the fixed point set for a finite group action on 
a CAT(0)-cube complex is contractible.  Since $G_n$ acts cellularly 
properly and cocompactly on the CAT(0)-cube complex $Y_n$, we deduce
that $Y_n$ is a cocompact model for $\underline{E}G_n$.  To summarize, we
have constructed a group $G=G_n$ with a cocompact classifying space
for proper actions $\underline{E}G$ admitting a compatible collection
of complex representations of the finite subgroups of $G$ that cannot
be realized as $G$-equivariant virtual complex vector bundle over
$\underline{E}G$.

We remark that Wolfgang L\"uck has shown us another quite different way 
to find a finite group $F$ and an $F$-CW-complex $X$ that satisfy 
Lemma~\ref{lemma: no virtual bundle}; any such pair could be used 
to construct a group with similar properties to the group $G=G_n$.

\section{Real vector bundles} \label{sec: real}
One could apply the techniques of the previous section in the real
setting to obtain a group $G$ with cocompact classifying space for
proper actions $\underline{E}G$ so that the edge homomorphism
\[  \varepsilon_{G}\col \mathrm{KO}^{0}_{G}(\underline{E}G) 
\rightarrow \lim_{G/H
  \in \mathcal{O}_{\mathcal{F}}G} RO(H) \] is not surjective. Here one
would need the real version of \cite[Theorem 5.1]{Jackowski}, which
also holds as explained in the paragraph below \cite[Theorem
  5.1]{Jackowski}.

 Instead we give an explicit description of a 
group $G$ that admits $\mathbb{R}^2$ as a cocompact model for 
$\underline{E}G$ and admits a compatible collection of real 
representations of its finite subgroups that cannot be realized 
as a real $G$-vector bundle over $\mathbb{R}^2$.  

We start by describing a related group $\Gamma$ that is a 
2-dimensional crystallographic group, or wallpaper group; this 
group is known as $p2gg$, but we will describe it explicitly.  
Endow $\mathbb{R}^2$ with the CW-structure coming from the 
standard tesselation by unit squares with vertices at $\mathbb{Z}^2$, 
and let $\Gamma$ be the group of automorphisms of this CW-structure
that preserves the pattern shown in Figure~\ref{figone}.  The stabilizer
of a 2-cell is clearly trivial, and so the 2-cells form a single 
free $\Gamma$-orbit.  There are two orbits of 1-cells, the vertical 
and horizontal edges, and again each orbit is free.  There are two 
orbits of 0-cells, and the stabilizer of a 0-cell is cyclic of 
order two, generated by the rotation of order two fixing the point.  
Since the stabilizer of each cell acts trivially on that cell, the 
given CW-structure makes $\mathbb{R}^2$ into a $\Gamma$-CW-complex.  

The translation subgroup $T$ of $\Gamma$ has index four, and consists 
of the elements $(x,y)\mapsto (x+2m,y+2n)$.  The orientation-preserving
subgroup $N$ of $\Gamma$ has index two, and consists of $T$ together
with the rotations through $\pi$ about some point of $\mathbb{Z}^2$, 
which are of the form $(x,y)\mapsto (2m-x,2n-y)$.  Finally the 
elements of $\Gamma - N$ are the glide reflections whose axes bisect
the sides of the 2-cells: $(x,y)\mapsto (2m+1-x,2n+1+y)$ and 
$(x,y)\mapsto (2m+1+x,2n+1-y)$.  The quotients $T\backslash \mathbb{R}^2$, 
$N\backslash \mathbb{R}^2$ and $\Gamma\backslash \mathbb{R}^2$ are 
respectively a torus consisting of four squares, an $S^2$ obtained by 
identifying the boundaries of two squares, and a copy of $\mathbb{R}P^2$ 
obtained by identifying the edges of a square in pairs.  The fact 
that $\Gamma -N$ contains no torsion elements is reflected in the 
fact that $\Gamma/N$ acts freely on the sphere $N\backslash \mathbb{R}^2$.  

Now let $F$ be a copy of $C_4$ and let $H\cong C_2$ be the index two 
subgroup of $F$.  The group $G$ is defined as the pullback of the two 
maps $\Gamma\rightarrow \Gamma/N\cong C_2$ and $F\rightarrow F/H\cong C_2$.  
By construction the group $G$ admits $\mathbb{R}^2$ as a cocompact 
model for $\underline{E}G$, and fits into a short exact sequence 
\[1 \rightarrow N \rightarrow G \xrightarrow{p} F \rightarrow 1\]
such that every finite subgroup of $G$ maps onto a subgroup of $H$ under $p$.\\

\begin{figure}
\begin{center}
\includegraphics[scale =1.1]{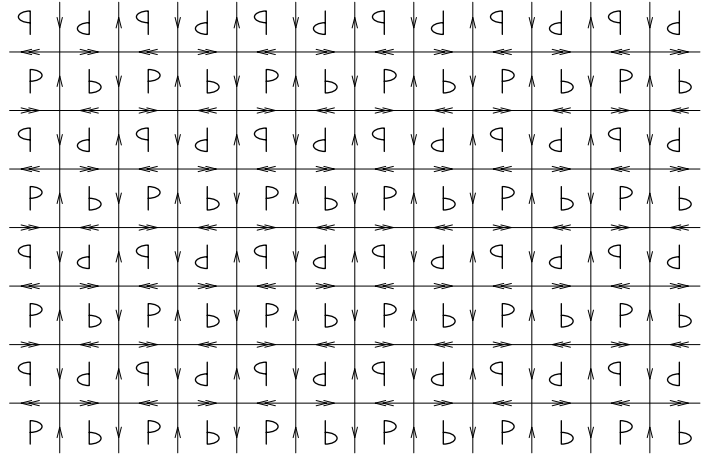}
\end{center}
\caption{\label{figone}A wallpaper pattern for $\Gamma=p2gg$} 
\end{figure}

Now let
\[\lambda: H \rightarrow O(1,\mathbb{R})=C_2\]
be the $1$-dimensional real sign representation of $H$, i.e.~$\lambda$
is the identity map. The isomorphism class $[\lambda]$ is clearly
contained in $RO(H)^{F/H}$, since $F$ is abelian.

\begin{lemma}\label{not in image} The isomorphism class $k[\lambda]$
  is contained in the image of the restriction map \[RO(F)\rightarrow
  RO(H)^{F/H}.\]  
if and only if $k$ is even. 
\end{lemma}
\begin{proof}

Recall that the irreducible real representations of $C_4$ are up to
isomorphism the one-dimensional trivial representation representation,
the one dimensional sign representation of $F/H=C_2$ and one 
2-dimensional faithful representation in which the elements of order 
four act as rotations by $\pm \frac{\pi}{2}$. The restriction of the
first two of the representations to $H$ gives the trivial one
dimensional representation of $H$, while the restriction to $H$ of the
third is $\lambda\oplus \lambda$. We therefore conclude that the
image of $RO(F)\rightarrow RO(H)^{F/H}$ consists of element of  the
form $2n[\lambda]+m[\mathrm{tr}]$, where $\mathrm{tr}$ is the trivial
one dimensional representation of $H$ and $n,m \in \mathbb{Z}$.  This
shows that $k[\lambda]$ is contained in the image of the restriction
map $RO(F)\rightarrow RO(H)^{F/H}$ if and only if $k$ is even. 
\end{proof}

\begin{lemma} \label{lemma: real line bundles}  Let $F$ act on the
  infinite dimensional sphere $S^{\infty}$ by first projecting onto
  $F/H=C_2$ and then acting via the antipodal map.  View $S^2$ as the
  $2$-skeleton of $S^{\infty}$. Every $F$-equivariant orthogonal real
  line bundle over $S^2$ is isomorphic to the pullback of an $F$-equivariant
  orthogonal real line bundle over $S^{\infty}$ along the inclusion $S^2
  \rightarrow S^{\infty}$. 

\end{lemma}

\begin{proof} Let $\mathcal{S}$ be the family of subgroups of $F$
  containing $H$ and the trivial subgroup. Note that isomorphism classes of $F$-equivariant
  orthogonal real line bundles correspond to isomorphism classes of $(F,C_2)$-bundles. Now let $\xi$ be an $(F,C_2)$-bundle
  over $S^2$ with fibers $A=(\xi_S)\in \lim_{S \in
    \mathcal{S}}\mathrm{Rep}_{C_2}(S)$. 
 By Lemma \ref{lemma: uni bun}, it suffices to show that every $F$-map $f\col S^2
\rightarrow \mathrm{B}_{\mathcal{S}}(F,A)$ can be extended to an
$F$-map $\tilde{f}\col S^{\infty} \rightarrow
\mathrm{B}_{\mathcal{S}}(F,A)$. Again by Lemma \ref{lemma: uni bun}, $\mathrm{B}_{\mathcal{S}}(F,A)^S \cong
BC_2=\mathbb{\mathbb{R}}P^{\infty}$ for all $S \in \mathcal{S}$.  It follows from Bredon's equivariant obstruction theory (see \cite[Section II.1]{Bredon},\cite[Th. I.5.1]{May}) that the potential obstructions for extending such a map lie in the relative Bredon cohomology groups
$\mathrm{H}^{n+1}_{F}(S^{\infty},S^2;\pi_{n}(\mathrm{B}_{\mathcal{S}}(F,A)^{-}))$
for $n\geq 2$. Since $\pi_n(\mathbb{R}P^{\infty})$ is zero unless
$n=1$, the lemma is proven.  
\end{proof}

\begin{lemma}\label{lemma: no real bundle} Let $F$ act on $S^2$ by
  first projecting onto $F/H=C_2$ and then acting via the antipodal
  map. There does not exist a real $F$-vector bundle $\xi\col E
  \rightarrow S^2$ such that the representation of $H$ on the fibers
  of $\xi$ is isomorphic to $\lambda$. 
\end{lemma}
\begin{proof} 
Consider the infinite dimensional sphere $S^{\infty}$ as a the
universal $C_2$-space $EC_2$, where $C_2$ acts via the antipodal map
and let $F$ act on $S^{\infty}$ via first projection onto $F/H=C_2$
and then acting via $C_2$.  Now assume that there exists a real
$F$-vector bundle $\xi\col E \rightarrow S^2$ such that the
representation of $H$ on the fibers of $\xi$ is isomorphic to
$\lambda$. By Lemma \ref{lemma: real line bundles} there 
exists a real $F$-vector bundle $\xi\col E \rightarrow S^{\infty}$ such
that the representation of $H$ on the fibers of $\xi$ is isomorphic to
$\lambda$. By pulling back this bundle along the inclusion $S^n \rightarrow S^{\infty}$, there also
exists a real $F$-vector bundle $\xi\col E \rightarrow S^{n}$ such
that the representation of $H$ on the fibers of $\xi$ is isomorphic to
$\lambda$, for every $n\geq 2$.

By the real version of \cite[Theorem 5.1]{Jackowski} (see comments
below \cite[Theorem 5.1]{Jackowski} ), there are maps
\[   \alpha_{n}\col   RO(F)/I^n  \rightarrow  \mathrm{KO}^0_F(S^n)    \] 
that induce a map of inverse systems from $\{ RO(F)/I^n \}_{n\geq 0}
$ to $ \{ \mathrm{KO}^0_F(S^n)\}_{n\geq 0}$ that in turn induces 
an isomorphism of
pro-rings. Here $I$ is the kernel of the restriction map $RO(F)
\rightarrow RO(H)$. This implies that for sufficiently large $n\geq 1$
there exists a map $\beta_{1}\col\mathrm{KO}^0_F(S^n) \rightarrow R(F)/I$
making the following diagram commute
\[\xymatrix{RO(F)/I^n \ar[r]^{\alpha_n} \ar[dd]  
& \mathrm{KO}^0_F(S^n)  \ar[ddl]^{\beta_{1}} \ar[dd] \ar[dr]^{\varepsilon_F} &\\
	& &  RO(H)^{F/H} \\
	RO(F)/I \ar[r]^{\alpha_1} &  
\mathrm{KO}^0_F(S^1). \ar[ur]^{\varepsilon_F} }\]
This shows that the image of the restriction map 
\[      RO(F) \rightarrow RO(H)^{F/H}   \]
coincides with the image of the edge homomorphism 
\[    KO^{0}_F(S^n) \rightarrow RO(H)^{F/H}, \]
implying that the $H$-representations coming from the fibers of any real $F$-vector bundle over $S^n$ can be extended to virtual $F$-representations. However, since $\lambda$ does not lie in the
image of $\mathrm{RO}(F) \rightarrow \mathrm{RO}(H)$ by Lemma \ref{not
  in image} we arrive at a contradiction and conclude that there does
not exist a real $F$-vector bundle $\xi\col E \rightarrow S^2$ such that
the representation of $H$ on the fibers of $\xi$ is isomorphic to
$\lambda$. 
\end{proof}
Consider the projection $p\col G \rightarrow F$ and the compatible system
of real orthogonal representations \[ ([\lambda \circ p_{|S}])_{S \in
  \mathcal{F}}\in \lim_{G/S \in
  \orb}RO(S)=\mathrm{H}^{0}_{G}(\underline{E}G,RO(-)),\] and assume
that there exists a real $G$-vector bundle $\xi\col E \rightarrow
\mathbb{R}^2$ that realizes it. Since the kernel of $p\col G \rightarrow
F$ is $N$, it follows from the lemma below and our observations above
that $N \setminus \xi\col N\setminus E \rightarrow N\setminus X$ is an
$F$-vector bundle over $S^2$, where $F$ acts on $S^2$ via projection
onto $F/H=C_2$, followed by the antipodal map. Morever, the
representation of $H$ on the fibers of $N \setminus \xi$ is by
construction exactly $\lambda$. This however contradicts Lemma
\ref{lemma: no real bundle}, so we conclude that there does not exist
a real $G$-vector bundle $\xi\col E \rightarrow \mathbb{R}^2$ that
realizes the compatible system of real orthogonal representations
$(\lambda \circ p_{|S})_{S \in \mathcal{F}}$.

\begin{lemma} \label{lemma: quotient}
	
	Let $G$ be any discrete group with normal subgroup $N$ and let $X$ be a proper $G$-CW-complex. If $\xi\col E\rightarrow X$ is a
  $G$-vector bundle over $X$ such
  that $N\cap G_x$ acts trivially on $\xi^{-1}(x)$ for every $x \in
  X$, then \[N \setminus \xi\col N\setminus E \rightarrow N\setminus X\]
  is a $G/N$-vector bundle over $N\setminus X$. 

\end{lemma}
\begin{proof} Denote the projection $G\rightarrow G/N=Q$ by $\pi$. Let
  us first consider the case where $\xi$ is trivial (trivial in the
  sense of \cite[Section 6.1.]{Luck}), i.e. assume $\xi$ is a pullback  
\[    \xymatrix{
G \times_H V \ar[r] & G/H \\
E \ar[u]^{r} \ar[r]^{\xi} & X \ar[u]^{p}  }\]
of the $G$-vector bundle $G\times_H V \rightarrow G/H$ along the
$G$-map $p\col X \rightarrow G/H$ where $H$ is some finite subgroup of $G$ and $V$ is a finite dimensional real $H$-representation such that $H\cap N$ acts trivially on $V$. Consider the
pullback diagram  
\[    \xymatrix{
Q \times_{\pi(H)} V \ar[r] & Q/\pi(H) \\
P \ar[u]^{w} \ar[r]^{q} & N\setminus X \ar[u]^{N \setminus p}  }\]
of the $Q$-vector bundle $Q\times_{\pi(H)} V \rightarrow Q/\pi(H)$
along the $Q$-map $N \setminus p\col N\setminus X \rightarrow Q/\pi(H)$.
We define the map 
\[  \psi\col N \setminus E \rightarrow P : \overline{(g,v,x)} \mapsto
(\pi(g),v,\overline{x}).  \] 
It is easy to check that $\psi$ yields a well-defined morphism of $Q$-equivariant bundles over $N \setminus X$. Moreover, since $\psi $ is a fiberwise linear map of $Q$-vector bundles that is a fiberwise isomorphism, it follows that $\psi$ is a homeomorphism.

Now consider the general case. Let $\overline{x} \in N \setminus
X$. Since $\xi\col E \rightarrow X$ is locally trivial, $x \in X$ has
an open $G$-neighourhood $U$ such that there is a $G$-map $p\col U
\rightarrow G/H$ where $H$ is finite subgroup of $G$ and $\xi_{|U}$ is (homeomorphic to) the
pullback 
\[    \xymatrix{
G \times_H V \ar[r] & G/H \\
\xi_{|U} \ar[u] \ar[r]^{\xi} & U \ar[u]^{p}  }\]
of the $G$-vector bundle $G\times_H V \rightarrow G/H$ along the $G$-map $p\col U \rightarrow G/H$. By the above, the quotient diagram
\[    \xymatrix{
Q \times_{\pi(H)} V \ar[r] & Q/\pi(H) \\
N\setminus \xi_{|U} \ar[u] \ar[r]^{N \setminus \xi} & N\setminus U \ar[u]^{N \setminus p}  }\]
is a pullback diagram. Since $N\setminus U$ is an open
$Q$-neighbourhood of $\overline{x}$, it
follows that $N \setminus \xi\col N\setminus E \rightarrow N\setminus X$
is a $Q$-vector bundle. 
\end{proof}

We finish this section by explaining how a similar approach to the one above can be used to produce a group $G$ admitting a
three dimensional cocompact model for $\underline{E}G$ that has a
compatible system of one-dimensional complex representations that
cannot be realized as a complex $G$-vector bundle over 
$\underline{E}G$. As in Section 3, let $F=C_4 \rtimes C_2$ be the dihedral group of order $8$ where $\sigma$ is a generator for  $C_4$. Let $H=\langle \sigma^2 \rangle$ be the center of $F$,  which has order two and denote the $3$-skeleton of the universal $F/H$-space $X=E(F/H)$ by $X^3$. We let $F$ act on $X$ and $X^3$ via the projection onto $F/H$. Consider the complex $1$-dimensional representation
\[\lambda\col H=\langle \sigma^2 \rangle \rightarrow 
\mathrm{U}(1)=S^{1}: \sigma^2 \mapsto -1. \]
By \cite[Th.~A \& Th.~8.3]{Leary} there exists a compact
$3$-dimensional locally CAT(0)-cubical complex $T_{X^3}$ equipped with
a free cellular $F/H$-action, an $F/H$-equivariant map $t_{X^3}\col
T_{X^3} \rightarrow X^3$ and an isometric cellular involution $\tau$
on $T_{X^3}$ that commutes with the $F/H$-action on $T_{X^3}$ and the
map $t_{X^3}$ such the induced $F/H$-equivariant map $$\langle \tau
\rangle \setminus T_{X^3} \rightarrow X^3$$ is a homotopy
equivalence. Note that $F/H$ acts freely on $\langle \tau \rangle
\setminus T_{X^3}$ since it acts freely on $X^3$. Hence $T_{X^3}$ is
also the $3$-skeleton of a universal $F/H$-space $Z$. So we may
continue assuming that $Z = X$ and $\langle \tau \rangle \setminus
T_{X^3}=X^3$.

Now let $Y$ be the univeral cover of $T_{X^3}$ and let $\Gamma$ be the
group of self-homeomorphism of $Y$ that lifts the action of $F/H
\oplus \langle \tau \rangle $ on $T_{X^3}$. Then $Y$ is a
3-dimensional CAT$(0)$-cubical complex on which $\Gamma$ acts
properly, compactly and cellularly. By construction there is a
surjection $\alpha\col \Gamma \rightarrow F/H \oplus \langle \tau \rangle
$ whose kernel $\mathrm{Ker}(\alpha)$ is the torsion-free group of
deck transformations of $Y \rightarrow T_{X^3}$. Let $\pi$ denote the
composition of $\alpha$ with the projection of $F/H \oplus \langle
\tau \rangle$ onto $F/H$. Since $F/H$ acts freely on $T_{X^3}$ and
every finite subgroup of $\Gamma$ must fix a point of $Y$ since $Y$ is
CAT$(0)$, it follows that every finite subgroup of $\Gamma$ is
contained in the kernel of $\pi$, which we denote by $N$. Now define
the group $G$ to be the pullback of $\pi\col \Gamma \rightarrow F/H$
along $F \rightarrow F/H$. Then $G$ acts on $Y$ via the quotient map
$G \rightarrow G/H = \Gamma$ that fits into the short exact sequence
\[ 1 \rightarrow N \rightarrow G \xrightarrow{p} F \rightarrow 1.  \]
such that $p$ maps all the finite subgroup of $G$ onto a finite
subgroup of $H$ and $N \setminus Y = X^3$.

Let $\mathcal{F}$ be the family of finite subgroups of $G$, note that
$Y$ is a three dimensional cocompact model for $\underline{E}G$ and
suppose that there exists a $G$-vector $\xi\col E \rightarrow Y$ whose
fibers give rise to the compatible system of representations \[
([\lambda \circ p_{|S}])_{S \in \mathcal{F}}\in \lim_{G/S \in
  \mathcal{O}_{\mathcal{F}}G}R(S).\]

Applying Lemma \ref{lemma: quotient}, we obtain an $F$-equivariant
complex line bundle $N \setminus \xi\col N \setminus E \rightarrow X$
such that the representation of $H$ on the fibers of $N \setminus \xi$
is isomorphic to $\lambda$. By Lemma \ref{lemma: line bundles}, this
bundle can be extended to an $F$-equivariant complex line bundle over
$X=E(F/H)$. We now continue in a similar fashion as in the proof of
Lemma \ref{lemma: no real bundle} to conclude that $[\lambda]$ is
contained in the image of the restriction map $R(F) \rightarrow
R(H)^{F/H}$, which contradicts Lemma \ref{lemma: image res}. We
conclude that the bundle $\xi$ cannot exist.
\section{Right angled Coxeter groups}

Let $\Gamma$ be a finite graph. We denote the vertex set of
$\Gamma$ by $S=V(\Gamma)$ and the set edges of $\Gamma $ by
$E(\Gamma)\subseteq V(\Gamma)\times V(\Gamma)$. The right angled
Coxeter group determined by $\Gamma$ is the Coxeter group $W$ with
presentation
\[       W = \langle S \ | \ s^2 \ \mbox{for all $s \in V(\Gamma)$ and
  \ }   (st)^2 \ \mbox{if $(s,t) \in E(\Gamma)$} \rangle .  \] 
Note that $W$ fits into the short exact sequence
\[   1 \rightarrow N \rightarrow W \xrightarrow{p} F = \bigoplus_{s\in
  S} C_2 \rightarrow 1\] where $p$ takes $s\in S$ to the generator of
the $C_2$-factor corresponding to $s$.  A subset $J \subseteq S$ is
called spherical if the subgroup $W_J=\langle J \rangle$ is finite
(and hence isomorphic to $\bigoplus_{s \in J} C_2$). The empty subset
of $S$ is by definition spherical. We denote the poset of spherical
subsets of $S$ orderded by inclusion by $\mathcal{S}$. If $J \in
\mathcal{S}$, then $W_J$ is called a spherical subgroup of $W$, while
a coset $wW_J$ is called spherical coset. We denote the poset of
spherical cosets, ordered by inclusion, by $W\mathcal{S}$. Note that
$W$ acts on $W\mathcal{S}$ by left multiplication, preserving the
ordering. The Davis complex $\Sigma$ of $W$ is the geometric
realization of $W\mathcal{S}$. One easily sees that $\Sigma$ is a
proper cocompact $W$-CW-complex. Since $\Sigma$ admits a complete
CAT(0)-metric such that $W$ acts by isometries, it follows that
$\Sigma$ is a cocompact model for $\underline{E}W$ (see
\cite[Th. 12.1.1 \& Th. 12.3.4]{DavisBook}). A consequence of this
fact is that every finite subgroup of $W$ is subconjugate to some
spherical subgroup of $W$. This implies that the group $N$ defined
above is torsion-free. Since the quotient space $W \setminus \Sigma$ is
homeomorphic to the geometric realization of the poset $\mathcal{S}$,
which is contractible since it has a minimal element, another
consequence is that the quotient $\underline{B}W=W\setminus
\underline{E}W$ is contractible.  We refer the reader to
\cite{DavisBook} for more details and information about these groups
and the spaces on which they act.\\

Let $\mathcal{F}$ be the family of finite subgroups of $W$. Given an
abelian group $A$, we denote by
\[\underline{A}\col \mathcal{O}_{\mathcal{F}}W \rightarrow \mathrm{Ab}\]
the trivial functor that takes all objects to $A$ and all morphism to
the identity map. One can verify that  
\begin{equation}  \label{eq: quotient cohomology}
  \mathrm{H}^{\ast}_W(\underline{E}W,\underline{A}) \cong
  \mathrm{H}^{\ast}(\underline{B}W,A). \end{equation} 
\begin{lemma}\label{lemma: const coef} Let $A=( [p_{|H}])_{H \in
    \mathcal{F}} \in \lim_{W/H \in
    \mathcal{O}_{\mathcal{F}}W}\mathrm{Rep}_{F}(H)$. For every $k\geq
  0$, the contravariant functor 
	\[    \mathcal{O}_{\mathcal{F}}W \rightarrow \mathrm{Ab}\col W/H
        \mapsto \pi_{k}(B_{\mathcal{F}}(W,A)^H)   \] 
	equals the trivial functor $\underline{\pi_k(BF)}$.
\end{lemma}
\begin{proof}
Let $EF$ be a contractible $F$-$CW$-complex with free $F$-action and
consider the product space $\underline{E}W \times EF$. This space
becomes a $(W\times K)-CW$-complex by letting $(w,f) \in W\times F$
act on $(x,y)\in\underline{E}W \times EF $ as 
\[   (w,f)\cdot (x,y) = (w\cdot x, p(w)f \cdot y). \]
One checks that with this action  $\underline{E}W \times EF$ is a
model for $E_{\mathcal{F}}(W,A)$, i.e.~$(\underline{E}W \times EF)^K$
is contractible when $K \in \mathcal{F}_A$ and empty otherwise. By
definition, it follows that $\underline{E}W \times BF$ is a model
$B_{\mathcal{F}}(W,A)$, where $W$ acts on trivially on the second
coordinate. Since $\underline{E}W^H$ is contractible for every $H \in
\mathcal{F}$, the lemma follows easily. 
	
\end{proof}
Let
$\Gamma$ be either the orthogonal group $\mathrm{O}(n,\mathbb{R})$ or
the unitary group $\mathbf{U}(n)$.
\begin{lemma} \label{lemma: splitting}
Every element of \[ \lim_{W/H \in
\mathcal{O}_{\mathcal{F}}W}\mathrm{Rep}_{\Gamma}(H)\] is of the form
$( [\lambda \circ p_{|H}])_{H \in \mathcal{F}}$ for some group
homomorphism $\lambda\col  F \rightarrow \Gamma$.
\end{lemma}

\begin{proof} Every finite subgroup $H$ of $W$ is isomorphic to a
  finite direct sum of $C_2$'s. Since every element of order $2$ in
  $\Gamma$ is conjugate in $\Gamma $ to a diagonal matrix with $\pm 1$
  on the diagonal and commuting matrices can be simultaneously
  diagonalized (e.g. see \cite[Th. 1.3.12]{horn}), it follows that the
  image of every homomorphism $H 
  \rightarrow \Gamma$ is conjugate to a finite subgroup of $\Gamma$
  consisting of diagonal matrices. Hence, every element of  $\lim_{W/H \in
    \mathcal{O}_{\mathcal{F}}W}\mathrm{Rep}_{\Gamma}(H)$ is of the
  form  
$( [\alpha_H])_{H \in \mathcal{F}}$ where $\alpha_H\col H \rightarrow
\Gamma$ is a homomorphism whose image lands in the finite abelian
subgroup of $\Gamma$ consisting of diagonal matrices. Since every
finite subgroup of $W$ is 
subconjugate to a spherical subgroup $W_J$, the compatibility of the
representations tells us that $( [\alpha_H])_{H \in \mathcal{F}}$ is
completely determined by the homomorphisms $\alpha_{\langle s \rangle
}\col \langle s \rangle \rightarrow \Gamma$, for $s \in S$.
Since the images of the $\alpha_{\langle s \rangle }$ are diagonal,
they commute.  Therefore, one can define the homomorphism 
\[    \lambda\col F=\bigoplus_{s\in S} C_2 \rightarrow \Gamma: 
  (\sigma_s)_{s \in S} \mapsto \sum_{s\in S} \alpha_{\langle s
  \rangle}(\sigma_s). \] 
The compatibility of the representations implies that
 $$( [\lambda \circ p_{|H}])_{H \in \mathcal{F}}=( [\alpha_H])_{H \in
  \mathcal{F}} ,$$ 
proving the lemma.  
\end{proof}

The following theorem applies to both complex and real representations
and vector bundles. 

\begin{theorem} Let $W$ be a right angled Coxeter group. Every
  compatible collection of representations of the finite subgroups of
  $W$ can be realized as a $W$-equivariant vector bundle over the
  Davis complex $\Sigma=\underline{E}W$. 
\end{theorem}

\begin{proof} Consider $A=( [p_{|H}])_{H \in \mathcal{F}} \in
  \lim_{W/H \in \mathcal{O}_{\mathcal{F}}W}\mathrm{Rep}_{F}(H)$. It
  follows from Lemma \ref{lemma: uni bun} that the existence of a
  $(W,A)$-bundle over $\Sigma$ follows from the existence a $W$-map
  $\Sigma \rightarrow B_{\mathcal{F}}(W,A)$. Since by Lemma
  \ref{lemma: const coef}, the contravariant functor 
\[   \pi_k( B_{\mathcal{F}}(W,A)^{-})\col
\mathcal{O}_{\mathcal{F}}(W)\rightarrow \mathrm{Ab}: W/H \mapsto
\pi_k( B_{\mathcal{F}}(W,A)^{H})   \] 
equals the trivial functor $\underline{\pi_k( BF) }$ for all $k\geq
0$, it follows from (\ref{eq: quotient cohomology}) and the
contractibility of $\underline{B}W$ that the Bredon cohomology
groups $$\mathrm{H}^{k+1}_W(\Sigma, \pi_k(
B_{\mathcal{F}}(W,A)^{-}))$$ are zero for all $k\geq 0$. Since there
certainly exists a $W$-map from the $0$-skeleton of $\Sigma$ to
$B_{\mathcal{F}}(W,A)$, it follows from Bredon's equivariant
obstruction theory that there exists a $W$-map $\Sigma \rightarrow
B_{\mathcal{F}}(W,A)$.  

Now consider a compatible collection of representations of the finite
subgroups of $W$. By Lemma \ref{lemma: splitting}, this collection is
of the form  
\[( [\lambda \circ p_{|H}])_{H \in \mathcal{F}}\in \lim_{W/H \in
  \mathcal{O}_{\mathcal{F}}W}\mathrm{Rep}_{\Gamma}(H) \] for some
group homomorphism $\lambda\col F \rightarrow \Gamma$. Letting $A=(
[p_{|H}])_{H \in \mathcal{F}}$, it follows from the above that there
exists a $(W,A)$-bundle $\xi\col E \rightarrow \Sigma $. If
$\Gamma=\mathrm{O}(n,\mathbb{R})$ then  
\[    \xi\col  E \times_{F} \mathbb{R}^n \rightarrow \Sigma  \]
is a real $W$-vector bundle over $\Sigma$ that realizes $( [\lambda
  \circ p_{|H}])_{H \in \mathcal{F}}$, and if $\Gamma=U(n)$ then  
\[    \xi\col  E \times_{F} \mathbb{C}^n \rightarrow \Sigma  \]
is a complex $W$-vector bundle over $\Sigma$ that realizes $ ([\lambda
  \circ p_{|H}])_{H \in \mathcal{F}}$. Here $F$ acts on $\mathbb{R}^n$
or $\mathbb{C}^n$ via the map $\lambda$. 
\end{proof}

\begin{lemma} If $W$ is a right angled Coxeter group, then
  $\mathrm{H}^{n}_W(\Sigma,R(-))=0$ for all $n>0$, and 
$\mathrm{H}^0_W(\Sigma,R(-))$ is free abelian of rank 
equal to the number of spherical subgroups of $W$.  
\end{lemma}

\begin{proof}
This is proven in much the same way as the corresponding result for
homology in \cite{Sanchez}.  In more detail, one uses the cubical
structure on $\Sigma$, in which there is one orbit of $n$-cubes with
stabilizer isomorphic to $(C_2)^n$ for each $n$-tuple of commuting
elements of $S$.  (For each $n\geq 0$, for each spherical subgroup
$W_J\cong (C_2)^n$ and for each $w\in W$, the subposet consisting of
all special cosets contained in $wW_J$ is order isomorphic to the
poset of faces of an $n$-cube.  Furthermore this isomorphism is 
equivariant for the stabilizer subgroup $wW_Jw^{-1}\cong (C_2)^n$, 
acting on the $n$-cube as the group generated by reflections in 
its coordinate planes.  The realizations of these subposets
are the cubes that make up the cubical structure on $\Sigma$.  For
more details concerning the cubical structure on $\Sigma$
see~\cite[Ch.~1.1-1.2~or~Ch.~7]{DavisBook}.)  Since the stabilizer of
a cube of strictly positive dimension acts non-trivially on the cube,
this cubical structure is not a $W$-CW-structure on $\Sigma$.
However, its barycentric subdivision is a simplicial complex 
naturally isomorphic to the realization
of the poset $W\mathcal{S}$ as described in the introduction to this
section.  

Let $\Sigma^n$ denote the $n$-skeleton of $\Sigma$ with the 
cubical structure.  Firstly, $\Sigma^0$ consists of a single 
free $W$-orbit of vertices, so $\mathrm{H}^*_W(\Sigma^0;R(-))$ is 
isomorphic to the ordinary cohomology of a point; since $W$ acts freely 
the calculation reduces to an equivariant cohomology calculation
for the trivial group action.  

Let $I=[-1,1]$ be an interval, with $C_2$ acting by $x\mapsto -x$
(i.e., swapping the ends of the interval).  Note that $I$ is 
equivariantly isomorphic to the Davis complex for the Coxeter group $C_2$.  
Let $\partial I$ denote the two end points $\{-1,1\}$.  Make 
$I$ into a $C_2$-CW-complex, for example by taking three 0-cells 
in two orbits at the points $-1$, $0$ and $1$, and one free orbit 
of 1-cells consisting of the two intervals $[-1,0]$ and $[0,1]$.  
The cellular $C_2$-Bredon cochain complex for the pair $(I,\partial I)$ 
with coefficients in $R(-)$ is a cochain complex of free abelian 
groups in which the degree zero term has rank two, the degree one
term has rank one, and all other terms are trivial.  
A direct computation with this cochain 
complex shows that $H^m_{C_2}(I,\partial I;R(-))$ is 
isomorphic to $\mathbb{Z}$ for $m=0$ and is zero for $m>0$.  

Next consider $I^n$ with $C_2^n$ acting as the direct product 
of $n$ copies of the above action of $C_2$ on $I$.  This is the 
Davis complex for the Coxeter group $C_2^n$.  Since the representation
ring of a direct product of finite groups is naturally identified with
the tensor product of the representation rings~\cite[Ch.~3.2]{serre}, 
the $C_2^n$-Bredon cochain complex for the pair $(I^n,\partial I^n)$ 
with coefficients in $R(-)$ is naturally isomorphic to the tensor 
product of $n$ copies of the $C_2$-Bredon cochain complex for 
$(I,\partial I)$ with coefficients in $R(-)$.  (If one wants to 
think about this cochain complex geometrically, it arises from 
the $(C_2)^n$-CW-structure on $I^n$ in which the cells are the 
direct products of the cells arising in the $C_2$-CW-structure 
on $I$.)  Since these cochain complexes consist of finitely 
generated free abelian groups, there is a K\"unneth 
formula as described in for example~\cite[Thrm~60.3]{munkres}.  
Since $H^*_{C_2}(I,\partial I;R(-))$ is free abelian the 
K\"unneth formula implies that 
\[H^*_{C_2^n}(I^n,\partial I^n,R(-))\cong 
\bigotimes_{i=1}^n H^*_{C_2}(I,\partial I;R(-)).\]
It follows that for each $n$, $H_{C_2^n}^m(I^n,\partial I^n;R(-))$ is 
isomorphic to $\mathbb{Z}$ for $m=0$ and is zero for $m>0$.  

From these computations, it follows easily that
$H_W^m(\Sigma^n,\Sigma^{n-1};R(-))$ is zero for $m>0$ and is 
isomorphic to a direct sum of copies of $\mathbb{Z}$ indexed by  
the $W$-orbits of $n$-cubes in $\Sigma$.  By induction on $n$ one 
sees that $H_W^m(\Sigma^n;R(-))$ is zero for $m>0$ and 
isomorphic to a direct sum of copies of $\mathbb{Z}$ indexed 
by the $W$-orbits of cubes of dimension at most $n$ for $m=0$.  
The claimed result follows, since the $W$-orbits of cubes in 
$\Sigma$ are in bijective correspondence with the spherical 
subgroups of $W$.  
\end{proof}

Before stating our theorem concerning
$\mathrm{K}^*_W(\underline{E}W)$, we make some remarks concerning the 
representation ring of a direct sum of copies of the cyclic group 
$C_2$, indexed by a (finite) set $S$.  For any finite group $G$, the 
collection of all isomorphism types of 1-dimensional complex
representations of $G$ is an abelian group, with product given by 
taking the tensor product of representations.  Furthermore, this 
group is naturally isomorphic to the group $\mathrm{Hom}(G,\mathrm{U}(1))$.  
In the case when $G$ is abelian, every irreducible representation 
of $G$ is 1-dimensional, and so $\mathrm{Hom}(G,\mathrm{U}(1))$ forms a 
basis for the additive group of the representation ring.  In this 
way the representation ring $R(G)$ is naturally isomorphic to the 
integral group algebra of the group $\mathrm{Hom}(G,\mathrm{U}(1))$.  
In the case 
when $G=\bigoplus_{s\in S}C_2$ is a direct sum of copies of $C_2$ indexed by 
$S$, we may view $G$ as a vector space over the field of two elements, 
in which case $\mathrm{Hom}(G,\mathrm{U}(1))$ may be identified with the dual 
space.  For $s\in S$, let $s^*$ denote the 1-dimensional
representation of $G$ with the properties that $s^*(s)=-1$ and
$s^*(t)=1$ for $t\in S-\{s\}$.  let $S^*$ denote the set of these 
representations: $S^*:=\{s^* \ | \ s\in S\}$.  In terms of vector 
spaces over the field of two elements, 
$S^*\subseteq \mathrm{Hom}(G,\mathrm{U}(1))$ 
is the dual basis to the set $S\subseteq G$.  
The set $S^*$ generates the representation ring of $G$, giving rise 
to the following presentation: 
\[R(G)= \mathbb{Z}[S^*]/({s^*}^2-1 \ | \ s\in S),\]
in which the monomials $s_{1}^*s_{2}^*\cdots s_{k}^*$ for 
all subsets $\{s_1,\ldots,s_k\}\subseteq S$ correspond to the 
irreducible representations.  

Suppose now that $J$ is a subset of $S$.  The inclusion $J\subseteq S$ 
identifies $H=\bigoplus_{s\in J}C_2$ with a subgroup of 
$G=\bigoplus_{s\in J}C_2$.  The induced map $R(G)\rightarrow R(H)$ of 
representation rings is described easily in terms of the above ring
presentation: for $s\in J$, $s^*\in R(G)$ restricts to $s^*\in R(H)$, 
while for $s\notin J$, $s^*\in R(G)$ restricts to $1\in R(H)$.  

Now suppose that $\Gamma$ is a graph with vertex set $V(\Gamma)=S$, 
and let $W$ be the right angled Coxeter group associated to $\Gamma$.  
The abelianization of $W$ is naturally identified with 
$G=\bigoplus_{s\in S}C_2$.  There is a unique equivariant map 
$\alpha\col\underline{E}W\rightarrow *$, from the $W$-space $\underline{E}W$ 
to a point $*$, viewed as a $G$-space with trivial action.  If $J$ is a 
spherical subset of $S$ then $W_J=\bigoplus_{s\in J}C_2$ maps 
isomorphically to the corresponding subgroup of 
$G=\bigoplus_{s\in S}C_2$.  If $x\in \underline{E}W$ is a 0-cell fixed 
by $W_J=\bigoplus_{s\in J}C_2$, then $\alpha(x)=*$, and this map is
$W_J$-equivariant.  The induced map 
$\alpha^*\col\mathrm{K}^*_G(*)\rightarrow \mathrm{K}^*_W(\underline{E}W)$, and the
composite map $\mathrm{K}^*_G(*)\rightarrow \mathrm{K}^*_{W_J}(\{x\})$ will be 
used in the statement and proof of our theorem.  If we identify 
$R(G)$ with $\mathrm{K}^0_G(*)$ and $R(W_J)$ with 
$\mathrm{K}^0_{W_J}(\{x\})$, then 
the composite is identified with the restriction map.  

\begin{theorem} \label{thm: equivk} 
Let $W$ be the right angled Coxeter group determined by a finite 
graph $\Gamma$, with vertex set $S$, and let $G=\bigoplus_{s\in S}$ 
be the abelianization of $W$.  The map 
$\alpha^*\col\mathrm{K}^*_G(*)\rightarrow 
\mathrm{K}^*_W(\underline{E}W)$ is surjective 
in each degree.  In particular, $\mathrm{K}^1_W(\underline{E}W)=0$ 
and there is a ring isomorphism
\[   \mathrm{K}^0_{W}(\underline{E}W) \cong 
 \mathbb{Z}[S^*]/ ({s^*}^2-1 , s^*t^*-s^*-t^*+1 \ |  \ \mbox{$s\in
   S=V(\Gamma)$,$(s,t)\notin E(\Gamma)$}). \]
It follows that $\mathrm{K}^{0}_{W}(\underline{E}W) \cong \mathbb{Z}^d$
as an abelian group, where $d$ is the number of spherical subgroups 
of~$W$. 
\end{theorem}


\begin{proof} Consider the Atiyah-Hirzebruch spectral sequence 
(\ref{eq: ss complex})
\[E_2^{p,q}=\mathrm{H}^{p}_W(\underline{E}W,\mathrm{K}^{q}_W(W/-))
\Longrightarrow \mathrm{K}_W^{p+q}(\underline{E}W)\] 
where $\mathrm{K}^{q}_W(W/-)=R(-)$ if $q$ is even and 
$\mathrm{K}^{q}_W(W/-)=0$ if $q$ is
odd (see \cite[Th. 3.2]{LuckOliver}). In the lemma above, we proved that
$\mathrm{H}^k_W(\Sigma,R(-))=0$ for $k>0$. It therefore follows that  
\[   \mathrm{K}^{n}_{W}(\underline{E}W) =\Big\{\begin{array}{cc}
\mathrm{H}^{0}_W(\underline{E}W,R(-)) = 
\lim_{W/H \in \mathcal{O}_{\mathcal{F}W}}R(H)& \mbox{if $n=0$} \\
0 & \mbox{if $n=1$.}
\end{array}   \]
Let $I$ be the ideal 
\[    ({s^*}^2-1 , s^*t^*-s^*-t^*+1 \ |  \ 
\mbox{$s\in S$, $(s,t)\notin E(\Gamma)$}) \]
in the polynomial ring $\mathbb{Z}[S^*]$.  Note that as an
abelian group $ \mathbb{Z}[S^*]/I$ is free, with basis elements
the commuting products $s^*_{1}\ldots s^*_k$, for all $J=\{s_1\ldots , s_k
\} \in \mathcal{S}$ (The case $J=\emptyset$ corresponds to the unit of $
\mathbb{Z}[V(\Gamma)]/I$ ). This shows that  
\[      \mathbb{Z}[S^*]/I \cong  \mathbb{Z}^d\] as an abelian
group, where $d$ is the number of spherical subgroups of $W$.

We claim there is an isomorphism of rings
\[  \lim_{W/H \in \mathcal{O}_{\mathcal{F}}W}R(H) \cong  \mathbb{Z}[S^*]/I. \]
Since every finite subgroup of $W$ is subconjugate to a spherical 
subgroup of $W$, it follows that 
\[ \lim_{W/H \in \mathcal{O}_{\mathcal{F}}W}R(H)\cong   \lim_{J \in \mathcal{S}}R(W_J) \]
as rings.  By the remarks in the paragraph preceeding the statement of 
the theorem, there are ring isomorphisms 
\[R(W_J)=\mathbb{Z}[J^*]/({s^*}^2-1\ | \ s\in J),\qquad 
R(G)=\mathbb{Z}[S^*]/({s^*}^2-1\ | \ s\in S),\] 
%
which are natural for inclusions $J\subseteq J'\subseteq S$.  
From this it follows that the natural ring homomorphism 
\[\rho\col R(G)\rightarrow \lim_{W/H\in \mathcal{O}_{\mathcal{F}W}} R(H)\] 
is surjective, and that $\lim_{W/H\in \mathcal{O}_{\mathcal{F}W}} R(H)$ 
is isomorphic to the ring described in the statement; in particular 
its additive group is free abelian of the same rank as
$\mathrm{K}^0_W(\underline{E}W)$.  Since $\rho$ factors through 
$\mathrm{K}^0_W(\underline{E}W)$, the claimed isomorphism follows.  
\end{proof}

Before stating our corollary concerning $\mathrm{K}^*(BW)$, we recall some 
facts from~\cite{atiyah} concerning $\mathrm{K}^*(BG)$, where as above 
$G=\bigoplus_{s\in S}C_2$.  For any finite group $H$, Atiyah showed
that $\mathrm{K}^i(BH)=0$ for $i$ odd, and that 
$\mathrm{K}^{2i}(BH)$ is naturally 
isomorphic to the completion of the representation ring $R(H)$ 
at its augmentation ideal.  To discuss the case of $G$, it is 
convenient to take new generators for $R(G)$; replace the irreducible 
representation $s^*$ by the degree zero virtual representation 
$\overline{s}=s^*-1$.  With respect to these generators one obtains 
the presentation 
\[R(G) = \mathbb{Z}[\overline{S}]/(\overline{s}(\overline{s}+2) \ |
\ s\in S),\] 
where $\overline{S}=\{\overline{s}\ | \ s\in S\}$.  If
$H=\bigoplus_{s\in J}C_2$, then of course there is a similar
description of $R(H)$, which is natural for the inclusion $J\subseteq
S$.  Note that if $s\notin J$, then the image of $\overline{s}$ under
the restriction map $R(G)\rightarrow R(H)$ is zero.

Completing $R(G)$, as described above, with respect to its
augmentation ideal gives rise to the following presentation for 
the ring $\mathrm{K}^0(BG)$: 
\[\mathrm{K}^0(BG) = 
\mathbb{Z}[[\overline{S}]]/(\overline{s}(\overline{s}+2) \ |
\ s\in S),\] 
which is natural for the inclusion $J\subseteq S$, and so also
describes the induced map $\mathrm{K}^0(BG)\rightarrow 
\mathrm{K}^0(BH)$.  The additive 
group of this ring is the direct sum of one copy of $\mathbb{Z}$,
generated by $1$, and for each non-empty subset $J\subseteq S$, one 
copy of the 2-adic integers, $\mathbb{Z}_2$, consisting of the set of 
power series in the element $\prod_{s\in J}\overline{s}$ with zero 
constant term.  

\begin{corollary} \label{cor: ktheory}  Let $W$ be the right angled
  Coxeter group determined by a finite graph $\Gamma$ with vertex set
  $S=V(\Gamma)$, and let $G=\bigoplus_{s\in S}C_2$ be the abelianization of
  $W$.  The induced map $\mathrm{K}^*(BG)\rightarrow 
\mathrm{K}^*(BW)$ is surjective in
  each degree.  In particular $\mathrm{K}^1(BW)=0$ and there is a ring
  isomorphism 
\[\mathrm{K}^0(BW)\cong
\mathbb{Z}[[\overline{S}]]/(\overline{s}(\overline{s}+2),\,\,
\overline{s}\overline{t} \ | \ s\in S,\,\,(s,t)\notin E(\Gamma)).\] 
%
Here, $\mathbb{Z}[[\overline{S}]]$ is the formal power series ring 
with $\mathbb{Z}$ coefficients in the variables 
$\overline{S}=\{\overline{s} \ | \ s\in S\}$.  
\end{corollary}

\begin{proof} The version of the Atiyah-Segal completion theorem 
that is proven for infinite discrete groups admitting a cocompact 
model for the classifying space for proper actions 
in~\cite[Theorem 4.4.(b)]{LuckOliver} implies that  
\[   \mathrm{K}^{n}(BW)=  \mathrm{K}^{n}_{W}(\underline{E}W) _{\hat{J}}, \]
where the ideal $J$ is the kernel of the augmentation map $
\mathrm{K}^{n}_{W}(\underline{E}W) \rightarrow \mathbb{Z}$ that maps
vector bundles to their dimension.  Changing variables in the 
above theorem to $\overline{s}=s^*-1$, we see 
that $\mathrm{K}^i(BW)=0$ for $i$ odd and that $\mathrm{K}^0(BW)$ is the 
completion of the ring 
\[\mathbb{Z}[\overline{S}]/(\overline{s}(\overline{s}+2),\,\, 
\overline{s}\overline{t}\ | \ s\in S,\,\,(s,t)\notin E(\Gamma))\] 
with respect to the ideal generated by the set
$\overline{S}=\{\overline{s}\ | \ s\in S\}$.  
This completion is the ring described in the statement. 
\end{proof}

There is an alternative proof of Corollary~\ref{cor: ktheory} that 
does not use Theorem~\ref{thm: equivk} or results
from~\cite{LuckOliver}.  Instead one uses a description of $W$ as 
a free product with amalgamation.  
If the graph $\Gamma$ is not a complete graph, then there 
is an expression $\Gamma=\Gamma_1\cup \Gamma_2$,
$\Gamma_3=\Gamma_1\cap \Gamma_2$, in which each $\Gamma_i$ is a full 
subgraph of $\Gamma$ and has fewer vertices than $\Gamma$.  This 
gives an expression for $W$ as a free product with amalgamation 
$W=W_1*_{W_3}W_2$.  From this one obtains a Mayer-Vietoris sequence 
that can be used to compute $\mathrm{K}^*(BW)$.  To establish 
Corollary~\ref{cor: ktheory}, one shows 
by induction on $|S|$ that 
$\mathrm{K}^*(BW)$ is as described and that for 
each $J\subseteq S$, the map $\mathrm{K}^*(BW)\rightarrow 
\mathrm{K}^*(BW_J)$ is a 
split surjection.

\end{document}